\documentclass[11pt]{amsart}
\pdfoutput=1

\usepackage{mathptmx}
\usepackage{amssymb, amsfonts, amsthm}
\usepackage{tikz}
\usepackage{pgf}
\usepackage{epstopdf}
\usepackage{booktabs}
\usepackage{multicol}
\usepackage{placeins}
\usepackage{caption} 
\usepackage{graphicx}
\usepackage{mathdots}
\usepackage{hyperref}
\usepackage{mathrsfs}
\usepackage{txfonts}
\usepackage{amsopn}
\usepackage{subfig}
\usepackage{checkend}
\usepackage{longtable,geometry}
\usepackage{physics}
\usepackage{etoolbox} 
\usepackage{xcolor}
\usepackage{cases}
\usepackage{scalefnt}

\newtheorem{theorem}{Theorem}
\newtheorem{definition}{Definition}
\newtheorem{lemma}{Lemma}
\newtheorem{remark}{Remark}

\begin{document}
	
\title[Double Degenerate Fast Diffusion]{Evolution of Interfaces for the Nonlinear Double Degenerate Parabolic Equation of Turbulent Filtration with Absorption. II. Fast Diffusion Case}

	\author[U. G. Abdulla]{Ugur G. Abdulla}
	\address{Department of Mathematical Sciences, Florida Institute of Technology, Melbourne, FL 32901}
	\email{abdulla@fit.edu}	
	
	\author[A. Prinkey]{Adam Prinkey}
           \author[M. Avery]{Montie Avery}

\maketitle
\begin{center}
Department of Mathematical Sciences\\ Florida Institute of Technology, Melbourne, FL 32901
\end{center}
\begin{abstract}
We prove the short-time asymptotic formula for the interfaces and local solutions near the interfaces for the nonlinear double degenerate reaction-diffusion equation of turbulent filtration with fast diffusion and strong absorption
\[ u_t=(|(u^{m})_x|^{p-1}(u^{m})_x)_x-bu^{\beta}, \, 0<mp<1, \, \beta >0. \]
Full classification is pursued in terms of the nonlinearity parameters  $m, p,\beta$ and asymptotics of the initial function near its support. In the case of an infinite speed of propagation of the interface, the asymptotic behavior of the local solution is classified at infinity. A full classification of the short-time behavior of the interface function and the local solution near the interface for the slow diffusion case ($mp>1$) was presented in {\it Abdulla et al., Math. Comput. Simul., 153(2018), 59-82}.  
\end{abstract}
\section{Introduction}
Consider the Cauchy problem (CP) for the nonlinear double degenerate parabolic equation
\begin{equation}\label{OP1}
Lu\equiv u_t-(|(u^{m})_x|^{p-1}(u^{m})_x)_x+bu^{\beta} = 0,  \, x\in \mathbb{R}, \, 0<t<T,
\end{equation}
\begin{equation}\label{IF1}
u(x,0)=u_0(x),\ x\in \mathbb{R}.
\end{equation}
where $u=u(x,t), m, p, \beta >0, b\in\mathbb{R},\, \text{with } 0<mp<1, \, \text{and } T\leq +\infty$ and $u_0$ is nonnegative and continuous. Throughout the paper we assume that either $b\geq 0$ or $b<0$ and $\beta\geq 1$ (see Remark~\ref{nonunique}).
Equation \eqref{OP1} arises in turbulent polytropic filtration of a gas in a porous medium \cite{Barenblatt1,Barenblatt2,Esteban,Leibenson}.  The condition $0<mp<1$ corresponds to 
the fast diffusion regime, when the equation \eqref{OP1}with $b=0$ possesses an infinite speed of propagation property \cite{Barenblatt1}.  
The main constituent of the equation \eqref{OP1} is to model competition between the double degenerate fast diffusion with infinite speed of propagation property and the absorption or reaction term. Assume that $\eta(0)=0$, where $\eta(\cdot)$ is an interface or free boundary defined as
\[\eta(t):=\text{sup}\{x:u(x,t)>0\}.\]
Furthermore, we shall assume that
\begin{equation}\label{IF2}
u_0(x)\sim C(-x)_+^{\alpha},\, \text{as } \, x\rightarrow 0^{-},\, \text{for some }  C>0,\, \alpha>0.
\end{equation}
where $(\cdot)_+=\max(\cdot; 0)$. Solution of the CP is understood in the weak sense. In Section~\ref{sec3}, we recall the definition of the weak solution (Definition~\ref{def: weak soln}) and the main results of the general theory.

 The aim of the paper is to classify short-time behavior of the interfaces and local solutions near the interfaces and at infinity in a CP with a compactly supported initial function. In all cases when $\eta(t)<+\infty$ we classify the short-time asymptotic behavior of the interface $\eta(\cdot)$, and local solution near $\eta(\cdot)$, while in all cases with $\eta(t)=+\infty$ we classify the short-time asymptotic behavior of the solution as $x\to+\infty$.  Classification is pursued in terms of parameters $m,p,b, \beta, C,$ and $\alpha$. 

Most of the results of the paper are local. Therefore, the behavior of $u_0(x)$ as $x\rightarrow -\infty$ is irrelevant, and we can assume that~$u_0$~ is either bounded or unbounded with growth condition as ~$x\rightarrow -\infty,$ which is suitable for the existence of the solution. In some cases we will consider the special case
\begin{equation}\label{IF3}
u_0(x)=C(-x)_+^{\alpha},\, x\in \mathbb{R},
\end{equation}
specifically when the solution to \eqref{OP1}, \eqref{IF3} is of self-similar form; in these cases the estimations will be global in time. 

A full classification of the small-time behavior of $\eta(t)$ and of the local solution near ~$\eta(t)$ depending on the parameters ~$m,p,b,\beta, C,$~and~$\alpha$ in the case of slow diffusion ($mp>1$) is presented in a recent paper \cite{AbdullaPrinkey1}. A similar classification for the reaction-diffusion equation \eqref{OP1} with $p=1$
is presented in \cite{Abdulla1} for the slow diffusion case ($m>1$), and in \cite{Abdulla3} for the fast diffusion case ($0<m<1$). The methods of the proof developed in \cite{Abdulla1,Abdulla3} are based on nonlinear scaling laws, and a barrier technique using special comparison theorems in irregular domains with characteristic boundary curves \cite{Abdulla4,Abdulla5,Abdulla6,Abdulla7}. Full classification of interfaces and local solutions near the interfaces and at infinity for the $p$-Laplacian type reaction-diffusion equation (\eqref{OP1} with $m=1$) are presented in \cite{AbdullaJeli1,AbdullaJeli2}. 
The semilinear case ($m=p=1$ in \eqref{OP1}) was analyzed in \cite{Grundy1,Grundy2}. It should be noted that the semilinear case is a singular limit of the general case. For instance, if $0<\beta<1,\; mp>\beta,\; \alpha<\frac{1+p}{mp-\beta}$, then the interface initially expands and if $mp>1$ then \cite{AbdullaPrinkey1}
\[\eta(t)\sim C_1t^{1/(1+p-\alpha(mp-1))}~~\text{as}~t\rightarrow 0^+,\]
while if $mp<1,$ we prove below that
\[\eta(t)\sim C_2t^{(mp-\beta)/[(1+p)(1-\beta)]}~~\text{as}~t\rightarrow 0^+.\]
Formally, as ~$m\rightarrow 1, p\rightarrow 1$ both estimates yield a false result, and from \cite{Grundy2} it follows that if $m=p=1$,~then 
\[\eta(t)\sim C_3(t~\text{ log}~ 1/t)^{\frac{1}{2}}\]
($C_i, i=\overline{1,3}$~ are positive constants).

The organization of the paper is as follows. In Section \ref{sec2} the main results are outlined, with further details in Section \ref{sec2a}. Essential lemmas are formulated and proven using nonlinear scaling in Section \ref{sec3}. Finally, in Section \ref{sec4}, the results of Sections \ref{sec2} are proved. To improve readability, explicit values of all constants that appear in Sections \ref{sec2}, \ref{sec2a}, and \ref{sec4} are relegated to the Appendix. 
 
 \begin{remark}\label{nonunique} The case $b<0, 0<\beta<1$ is not considered in this paper due to the fact that in general, uniqueness and comparison theorems don't hold for the solutions of the Cauchy problem \eqref{OP1},\eqref{IF1}. It should be pointed out that the methods of this paper can be applied to identify asymptotic properties of the minimal solution at infinity in this case. The methods of this paper can be applied to similar problem for the non-homogeneous reaction-diffusion equations with space and time variable dependent power type coefficients (\cite{shmarev2015interfaces}). It should be also mentioned that modification of the method can be applied to radially symmetric solutions of the multidimensional double degenerate reaction-diffusion equation
\[ u_t= div(|\nabla u^m|^{p-1}\nabla u^m)+bu^\beta. \]
\end{remark}

\section{Description of the main results}\label{sec2}

Throughout this section we assume that $u$ is a unique weak solution of the CP \eqref{OP1}-\eqref{IF2}. There are five different subcases, as shown in Figure \ref{Figure1}. The main results are outlined below in Theorems \ref{theorem 1}\,-\,\ref{theorem 5} corresponding directly to the cases I-V, respectively, in Figure \ref{Figure1}.
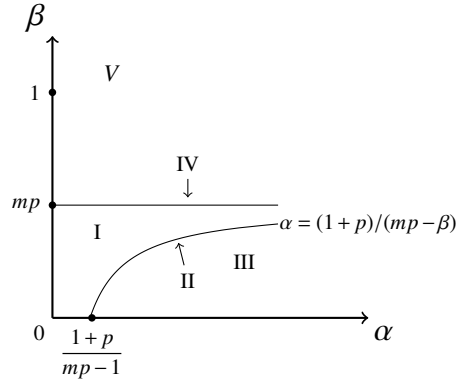
\begin{figure}[h!]
		\begin{center}	
		{\scalefont{0.75}
\begin{tikzpicture}[xscale=1.5, yscale=0.5]
\draw[->, thick] (0,0)--(2.8,0);
\draw[->,thick] (0,0)--(0,7.5);
\node[below left] at (0,0) {$0$};
\node[fill, shape=circle, label=180:{$mp$}, inner sep=1pt] at (0,3) {};
\node[fill, shape=circle, label=180:{$1$}, inner sep=1pt] at (0,6) {};
\node[fill, shape=circle, label=-90:{$\displaystyle  \frac{1+p}{mp-1}$}, inner sep=1pt] at (.35,0) {};
\node[left] at (0,8) {\scalefont{1.5} $\beta$};
\node[below] at (2.9,0) {\scalefont{1.5} $\alpha$};
\draw[domain=0:2.5, smooth] plot ({1/(3-\x)},\x);
\node at (2.8,2.5) {$\alpha =(1+p)/(mp-\beta)$};
\node at (0.5,6.5) {\scalefont{1.1} $V$};
\node (1) at (0.4,2.3) {I};
\node (2) at (1.2,1) {II};
\node (3) at (1.7,1.5) {III};
\node (4) at (1.2,4.1) {IV};
\node (*) at (1.2,3) {};
\node (**) at (.71,1.5) {};
\node (***) at (1.1,2.2) {};
\draw[ ->][ left]
  (4) edge (*);
   \draw[ ->][ left]
  (2) edge (***);
\draw (0,3)--(2,3);
\draw ( (0.5,1) {};
\end{tikzpicture}
}
		\caption{Classification of different cases in the ($\alpha$,$\beta$) plane for interface development in problem \eqref{OP1}-\eqref{IF2}.}
		\label{Figure1}
		\end{center}
		\end{figure}
\begin{theorem}\label{theorem 1} 
Let $b>0, 0 < \beta < mp$ and $0 < \alpha < (1+p)/(mp-\beta)$. The interface initially expands and there exists a number $\delta > 0$ such that
\begin{equation}
\zeta_1 t^\frac{mp-\beta}{(1+p)(1-\beta)} \leq \eta(t) \leq \zeta_2 t^\frac{mp-\beta}{(1+p)(1-\beta)}, \, 0 \leq t \leq \delta, \label{result2}
\end{equation}
(see the Appendix for explicit values of $\zeta_1$ and $\zeta_2$). Moreover, for any $\rho \in \mathbb{R}$, there is a number $f(\rho) > 0$ (depending on $C$, $m$, and $p$)  such that
\begin{equation}\label{sim1}
u(\xi_\rho (t), t) \sim f(\rho) t^\frac{\alpha}{1+p - \alpha (mp -1)},\, \text{as } t\rightarrow 0^{+},
\end{equation}
where $\xi_\rho (t) = \rho t^\frac{1}{1+p - \alpha (mp-1)}$.
\end{theorem}
\begin{theorem}\label{theorem 2} 
Let $b>0, 0<\beta<mp, \alpha=(1+p) /(mp-\beta)$, and 
\[C_*=\bigg[\frac{b(mp-\beta)^{1+p}}{(m(1+p))^{p}p(m+\beta)}\bigg]^{\frac{1}{mp-\beta}}.\]
Then the interface expands or shrinks accordingly as $C>C_*$ or $C<C_*$ and
\begin{equation}\label{eta3}
\eta(t)\sim\zeta_*t^{\frac{mp-\beta}{(1+p)(1-\beta)}},\, \text{as } t\rightarrow 0^{+},
\end{equation}
where $\zeta_* \lessgtr 0$ if $C \lessgtr C_*$, and for arbitrary $\rho < \zeta_*$ there exists $f_1(\rho)>0$ such that
\begin{equation}\label{sss3}
u(\zeta_\rho (t),t)\sim t^{1/(1-\beta)}f_1(\rho),\, \text{as } t\rightarrow 0^{+},
\end{equation}
where $\zeta_\rho (t) = \rho t^{\frac{mp-\beta}{(1+p)(1-\beta)}}$.
\end{theorem}
\begin{theorem}\label{theorem 3} 
Let $b>0, 0 < \beta < mp$ and $\alpha > (1+p)/(mp-\beta)$. Then the interface initially shrinks and 
\begin{equation}\label{etasim3}
\eta (t) \sim -\ell_* t^\frac{1}{\alpha (1-\beta)},\, \text{as } t\rightarrow 0^{+},
\end{equation}
where $\ell_* = C^{-1/\alpha} (b(1-\beta))^\frac{1}{\alpha (1-\beta)}$. For any $\ell > \ell_*$, we have
\begin{equation}\label{usim3}
\begin{aligned}
u(\eta_{\ell} (t),t) \sim \left[ C^{1-\beta} \ell^{\alpha (1-\beta)} - b (1-\beta) \right]^\frac{1}{1-\beta}t^\frac{1}{1-\beta}, \, \text{as } t\rightarrow 0^{+},
\end{aligned}
\end{equation} 
where $\eta_{\ell} (t) = -\ell t^{\frac{1}{\alpha(1-\beta)}}$.
\end{theorem}
\begin{theorem}\label{theorem 4} 
Let $b>0, \ \beta=mp$ and $\alpha>0$. In this case there is an infinite speed of propagation. For arbitary $\epsilon>0$, there exists a number $\delta = \delta(\epsilon) > 0$ such that
\begin{equation}\label{MR41}
t^\frac{1}{1-mp} \phi(x) \leq u(x,t) \leq  (t+\epsilon)^\frac{1}{1-mp}\phi(x), \, x>0, \, 0 \leq t \leq \delta,
\end{equation}
where $\phi=\phi(x)>0$ is a solution of the stationary problem
\begin{equation}
\begin{cases}
(|(\phi^m)'|^{p-1} (\phi^m)')' = \frac{1}{1-mp} \phi + b \phi^{mp}, \,x>0, \label{phiODE1} \\
\phi(0) = 1, \, \phi(+\infty) = 0.
\end{cases}
\end{equation}
Moreover, we have
\begin{equation}\label{u41}
\ln u(x,t) \sim -\frac{1}{m}\bigg(\frac{b}{p}\bigg)^{1/(1+p)}x, \text{ as } x\to+\infty, \, 0 \leq t \leq \delta.
\end{equation}
\end{theorem}
\begin{theorem}\label{theorem 5} 
Let either $b>0, \, \beta > mp$ or $b<0, \, \beta \geq 1$ or $b=0$, and
\[D=\bigg[\frac{(m(1+p))^p(m+1)}{(1-mp)^p}\bigg]^\frac{1}{1-mp}.\]
Then there is an infinite speed of propagation of the interface and \eqref{sim1} holds. 
If $b>0, \, \beta \geq \frac{p(1-m)+2}{1+p}$ or $b<0, \, \beta \geq 1$ or $b=0$, then there exists a number $\delta > 0$ such that 
\begin{equation}\label{asymp1}
u(x,t) \sim D t^\frac{1}{1-mp} x^\frac{1+p}{1-mp}, \text{ as } x \to +\infty, \, t \in (0,\delta],
\end{equation}
while if $b>0$ and $1 \leq \beta < \frac{p(1-m)+2}{1+p}$, then
\begin{equation}\label{asymp2}
\lim_{t \to 0^+} \lim_{x \to +\infty} \frac{u(x,t)}{t^\frac{1}{1-mp} x^\frac{1+p}{mp-1}} = D.
\end{equation}
If $b>0$ and $mp<\beta<1$, then there exists a number $\delta > 0$ such that 
\begin{equation}\label{asymp3}
u(x,t) \sim C_* x^\frac{1+p}{mp-\beta}, \text{ as } x \to +\infty, \, t\in(0, \delta].
\end{equation}

\end{theorem}
\section{Further details of the main results}\label{sec2a}
In this section we outline some essential details of the main results described in Theorems \ref{theorem 1}\,-\,\ref{theorem 5} of Section \ref{sec2}. We refer to the Appendix for the explicit values of relevant constants that appear throughout this section.\\
{\it Further details of Theorem \ref{theorem 1}}.
The solution $u$ satisfies the estimation
\begin{equation}
C_1 t^\frac{1}{1-\beta} (\zeta_1 - \zeta)^\frac{1+p}{mp-\beta}_+ \leq u(x,t) \leq C_* t^\frac{1}{1-\beta} (\zeta_2 - \zeta)^\frac{1+p}{mp-\beta}_+, \, 0 < t \leq \delta, \label{result1}
\end{equation}
where $\zeta = xt^\frac{\beta - mp}{(1+p)(1-\beta)}$. 
The left-hand side of \eqref{result1} is valid for $0\leq x <+\infty$, while the right-hand side is valid for $x\geq \ell_0 t^{(mp-\beta)/(1+p)(1-\beta)}$. $C_1, \, \zeta_1, \, \zeta_2$, and $\ell_0$, are positive constants depending on $m, \, p, \, \beta$, and $b$. 
Moreover,
\begin{equation}\label{f1}
f(\rho) = C^\frac{1+p}{1+p - \alpha (mp -1)} f_0 (C^\frac{mp -1}{1+p - \alpha (mp-1)} \rho), \ f_0 (\rho) = w(\rho, 1)\, \ \rho \in \mathbb{R},
\end{equation}
where $w$ is a minimal solution of the CP \eqref{OP1}, \eqref{IF3} with $C=1,\,  b=0$. If $u_0$ is given by \eqref{IF3}, then the right-hand sides of \eqref{result1} and \eqref{result2} are valid for all $t>0$. \\
{\it Further details of Theorem \ref{theorem 2}}.
Assume $u$ solves the CP \eqref{OP1}, \eqref{IF3}. 
If $C=C_*$, then $u_0$ is the stationary solution to the CP. If $C \neq C_*$, then the minimal solution of the CP is given by 
\begin{equation}
u(x,t)=t^\frac{1}{1-\beta} f_1(\zeta),\,  \zeta = xt^\frac{\beta - mp}{(1+p)(1-\beta)}, \label{sss2}
\end{equation}
and 
\begin{equation}
\eta (t) = \zeta_* t^\frac{mp-\beta}{(1+p)(1-\beta)}, \, t \geq 0, \label{eta2}
\end{equation}
If $C>C_*$, the interface initially expands and we have
\begin{gather}
C' (\zeta' t^\frac{mp-\beta}{(1+p)(1-\beta)} -x)^\frac{1+p}{mp-\beta}_+ \leq u \leq C'' (\zeta'' t^\frac{mp-\beta}{(1+p)(1-\beta)} -x)^\frac{1+p}{mp-\beta}_+, \label{result3} \\
\zeta' \leq \zeta_* \leq \zeta'', \, 0 \leq x < +\infty, \, t>0 \label{result4}
\end{gather}
where $C'=C_2, C''=C_*, \zeta'=\zeta_3, \text{ and } \zeta''=\zeta_4$.  
If $0 < C < C_*$, then the interface shrinks and there exists $\ell_1 >0$ such that for all $\ell \leq \ell_1$ there exists a number $\lambda >0$ such that
\begin{equation}\label{e2l3}
u(\ell t^\frac{mp-\beta}{(1+p)(1-\beta)} , t) = \lambda t^\frac{1}{1-\beta}, \, t \geq 0,
\end{equation}
and $u$ and $\zeta_*$ satisfy the estimates \eqref{result3}, \eqref{result4} with $C' = C_*, C'' = C_3, \zeta' = -\zeta_5, \text{ and }  \zeta'' = -\zeta_6$.\\
{\it Further details of Theorem \ref{theorem 3}}.
The interface initially coincides with that of the solution
\[\bar u(x,t)=\big [C^{1-\beta}(-x)_+^{\alpha(1-\beta)}-b(1-\beta)t\big]_+^{1/(1-\beta)}\]
to the problem 
\[\bar u_t+b\bar u^{\beta}=0,~~~~~~~~~~~\bar u(x,0)=C(-x)_+^{\alpha}.\]
{\it Further details of Theorem \ref{theorem 4}}.
The explicit solution of the problem \eqref{phiODE1} is given by
\begin{equation}\label{ODEsolution1}
\phi(x)=F^{-1}(x), \, 0\leq x < +\infty,
\end{equation}
where $F^{-1}(\cdot)$ is the inverse of the function
\begin{equation}\label{ODEsolution2}
F(z) = \displaystyle\int_{z}^{1} ms^{-1} \left[ \frac{b}{p} + \frac{m(1+p)}{p(1-mp)(1+m)} s^{1-mp} \right]^{-\frac{1}{1+p}}ds, \ 0<z\leq 1.
\end{equation}
The function $\phi(x)$ satisfies 
\begin{equation}\label{phi41}
\ln \phi(x) \sim -\frac{1}{m}\bigg(\frac{b}{p}\bigg)^{1/(1+p)}x, \text{ as } x\to+\infty,
\end{equation}
and the global estimation
\begin{equation}\label{phiest1}
0<\phi(x)\leq \exp\left(-\frac{1}{m} \left( \frac{b}{p} \right)^\frac{1}{1+p}x\right), \, x>0,
\end{equation}
and therefore
\begin{equation}\label{philim2}
\frac{\phi(x)}{e^{-\gamma x}} \rightarrow +\infty, \text{ as } x\rightarrow +\infty, \text{ if } \gamma > \frac{1}{m} \left( \frac{b}{p} \right)^\frac{1}{1+p}.
\end{equation}
From \eqref{MR41} and \eqref{philim2}, it follows that
\begin{equation}\label{uest1}
\lim_{t \to 0^+} \lim_{x \to +\infty} u(x,t)\exp\left(-\frac{1}{m} \left( \frac{b}{p} \right)^\frac{1}{1+p}x\right) = 0,
\end{equation}
and respectively
\begin{equation}\label{ulim2}
\frac{u(x,t)}{e^{-\gamma x}} \rightarrow +\infty, \text{ as } x\rightarrow +\infty, \, 0 \leq t \leq \delta(\epsilon), \text{ if } \gamma > \frac{1}{m} \left( \frac{b}{p} \right)^\frac{1}{1+p}.
\end{equation}
{\it Further details of Theorem \ref{theorem 5}}.
If $\beta \geq 1$, then for arbitrary $\epsilon >0$, there exists $\delta = \delta (\epsilon) > 0$ such that
\begin{equation}
C_5 t^\frac{\alpha}{1+p - \alpha(mp-1)} (\xi_1 + \xi)^\frac{1+p}{mp-1} \leq u(x,t) \leq C_6 t^\frac{\alpha}{1+p - \alpha (mp-1)} (\xi_2 + \xi)^\frac{1+p}{mp-1}, \label{result5}
\end{equation}
where $\xi = xt^\frac{-1}{1+p-\alpha(mp-1)}$, for all $x \in [0, \infty)$ and $0 \leq t \leq \delta(\epsilon)$.  $C_5, \, C_6, \, \xi_1$, and $\xi_2$, are positive constants depending on $m, \, p, \, \beta$, $b$, and $\epsilon$.
If $b>0$ and $\beta\geq1$, we have the upper estimation
\begin{equation}
u(x,t) \leq D t^\frac{1}{1-mp} x^\frac{1+p}{mp-1}, \, 0 < x < +\infty, \, 0 < t < +\infty.  \label{diffbound}\\
\end{equation}
If $b<0$ and $\beta \geq 1$, then for small $\epsilon > 0$, there exists $\delta = \delta(\epsilon) >0$ such that
\begin{equation}\label{fasymp1}
u(x,t) \leq D(1-\epsilon)^\frac{1}{mp-1} t^\frac{1}{1-mp} x^\frac{1+p}{mp-1}, \text{ for } \mu t^\frac{1}{1+p+\alpha (1-mp)} < x < +\infty, \, 0 < t \leq \delta, 
\end{equation}
where
\begin{equation}
\mu = \left[ \frac{A_0 + \epsilon}{D(1-\epsilon)^\frac{1}{mp-1}} \right]^\frac{mp-1}{1+p}. \label{mu1}\\
\end{equation}
If $b>0$ and $mp<\beta<1$, then there exists $\delta >0$ such that
\begin{equation}
t^\frac{1}{1-\beta} C_* (1-\epsilon) (\zeta_8 + \zeta)^\frac{1+p}{mp-\beta} \leq u(x,t) \leq C_* x^\frac{1+p}{mp-\beta}, \, 0 < x < +\infty, \, 0 < t \leq \delta, \label{mpbeta1}
\end{equation}
where $\zeta = xt^\frac{\beta - mp}{(1+p)(1-\beta)}$, $\epsilon>0$ is an arbitrary sufficiently small number, and  $\zeta_8$  is a positive constant depending on $m, \, p, \, \beta$, $b$, and $\epsilon$. 

If $b=0$ and $\alpha >0$, then the minimal solution to the CP \eqref{OP1}, \eqref{IF3} has the self-similar form
\begin{equation}\label{sss0}
u(x,t) = t^\frac{\alpha}{1+p + \alpha (1-mp)} f(\xi), \, \xi = xt^\frac{-1}{1+p+\alpha(1-mp)}.
\end{equation}
where $f$ satisfies \eqref{f1}. Moreover the following global estimation is valid:
\begin{gather}\label{ubb0}
D t^\frac{\alpha}{1+p + \alpha (1-mp)} (\xi_3 + \xi)^\frac{1+p}{mp-1} \leq u(x,t) \leq C_7 t^\frac{\alpha}{1+p + \alpha (1-mp)}  (\xi_4 + \xi)^\frac{1+p}{mp-1}, \\ 0 \leq x < +\infty, \, 0 < t < +\infty, \nonumber
\end{gather}
where  $C_7, \, \xi_3$, and $\xi_4$, are positive constants depending on $m$ and $p$. 

The right-hand side of \eqref{ubb0} is not sharp enough as  $x\rightarrow +\infty$ and the required upper estimation is provided by an explicit solution to \eqref{OP1}, as in  \eqref{fasymp1}. From \eqref{ubb0} and \eqref{fasymp1} it follows that, for arbitrary fixed $0<t<+\infty$, the asymptotic result \eqref{asymp1} is valid.
Now assume that $u_0$ satisfies \eqref{IF2} with $\alpha>0$. Then \eqref{sim1} is valid and for an arbitrary sufficiently small $\epsilon>0$ there exists a $\delta=\delta(\epsilon)>0$ such that the estimation \eqref{ubb0} is valid for $0<t\leq \delta$, except that in the left-hand side (respectively in the right-hand side ) of \eqref{ubb0} the constant $A_0$ should be replaced by $A_0-\epsilon$ (respectively $A_0+\epsilon$). Moreover, there exists a number $\delta>0$ (which does not depend on $\epsilon$) such that, for arbitrary $t\in(0,\delta]$, the asymptotic result \eqref{asymp1} is valid. 
\section{Preliminary results}\label{sec3}

The prelude of the mathematical theory of the nonlinear degenerate parabolic begins with the papers
\cite{zeldovich,Barenblatt1}, where instantaneous point source type particular solutions were constructed and analyzed. The property of finite speed of propagation and the existence of compactly supported nonclassical solutions and interfaces became a motivating force of the general theory.
The mathematical theory of nonlinear degenerate parabolic equations began with the paper \cite{Oleinik} on the porous medium equation (\eqref{OP1} with $p=1$). Currently there is a well established general theory of the nonlinear degenerate parabolic equations (see \cite{Vazquez2,Dibe-Sv,Abdulla6,Abdulla7,Abdulla8,Abdulla9,Abdulla11,Abdulla12,Abdulla14,Abdulla16,Abdulla18,Abdulla20,Abdulla21,Abdulla25,Herrero_Vazquez,HerreroVazquez1,Kalashnikov1,Kalashnikov4,Kersner1,Vasquez1,Gala1,shmarev}). Boundary value problems for \eqref{OP1} have been investigated in \cite{Kalashnikov2,Kalashnikov3,Esteban,Tsutsumi,Ishige,Degtyarev, Ivanov1, Ivanov2,Vespri1}.
\begin{definition}[Weak Solution]\label{def: weak soln}
A continuous nonnegative function $u(x,t)$ defined in $\mathbb{R} \times [0,T)$ is a weak solution of (\ref{OP1}), (\ref{IF1}) if for any $T_1\in (0,T)$ and any bounded interval
$(a,b)\in \mathbb{R}$, $(u^m)_x\in L^{p+1}((a,b)\times(0,T_1))$ and
\begin{equation}\label{weaksolution}
\int_0^{T_1}\int_a^b\Big(-u\phi_t+|(u^m)_x|^{p-1}(u^m)_x\phi_x+bu^\beta\phi\Big)dxdt=\int_a^bu\phi\Big |_{t=T_1}^{t=0}dx
\end{equation}
for arbitrary $\phi\in C^1([a,b]\times[0,T_1])$ such that $\phi \big|_{x=a}=\phi \big|_{x=b}=0$.
\end{definition}
If $u_0\in C(\mathbb{R})\cap L^\infty(\mathbb{R})$ and is nonnegative, then the 
existence, uniqueness, and comparison theorems for the weak solution of the CP  \eqref{OP1}), \eqref{IF1} have been proved in \cite{Esteban} for the case $b=0$, and in \cite{Tsutsumi} for $b> 0$. In \cite{Esteban} it is proved that the weak solution of \eqref{OP1}, $b=0$, is locally H\"{o}lder continuous. Local H\"{o}lder continuity of the locally bounded weak solutions of the general second order multidimensional nonlinear degenerate parabolic equations with double degenerate diffusion term is proved in \cite{Ivanov1,Ivanov2}. The following is the standard comparison result, which is widely used throughout the paper.
\begin{lemma}\label{CT}
Let $g$ be a non-negative and continuous function in $\overline{Q}$, where:
\[Q = \{(x,t): \eta_0(t) < x < +\infty, \, 0 < t  < T \leq +\infty \}\]
$g = g(x,t)$ is in $C^{2,1}_{x,t}$ in $Q$ outside a finite number of curves: $x = \eta_j(t)$, which divide $Q$ into a finite number of subdomains: $Q^j$, where $\eta_j \in C[0,T]$; for arbitrary $\delta > 0$ and finite $\Delta \in (\delta, T]$ the function $\eta_j$ is absolutely continuous in $[\delta, \Delta]$. Let $g$ satisfy the inequality:
\[ Lg\equiv g_t-(|(g^{m})_x|^{p-1}(g^{m})_x)_x+bg^{\beta} \geq 0, (\leq 0),\]
at the points of $Q$ where $g \in C^{2,1}_{x,t}$. Also assume that the function: $|(g^{m})_x|^{p-1}(g^{m})_x$ is continuous in $Q$ and $g \in L^{\infty}(Q \cap (t \leq T_1))$ for any finite $T_1 \in (0,T]$. If in addition we have that:
\[g(\eta_0(t),t) \geq  (\leq) \, u(\eta_0(t),t), \,\, g(x,0) \geq (\leq) \, u(x,0), \]
then
\[g \geq (\leq) \, u, \, \, \text{in} \, \, \overline{Q}\]
\end{lemma}

Suppose that $u_0\in C(\mathbb{R})$, and may have unbounded growth as $|x|\rightarrow +\infty$. It is well known that in this case some restriction must be imposed on the growth rate for the existence and uniqueness of the solution to the CP \eqref{OP1}, \eqref{IF1}. For the particular cases of the equation \eqref{OP1} with $b=0$, this question was settled in \cite{BCP,HerreroPierre} for the porous medium equation ($p=1$) with slow ($m>1$) and fast ($0<m<1$) diffusion; and in \cite{DibeHerrero1,DibeHerrero2} for the $p$-Laplacian equation ($m=1$) with slow ($p>1$) and fast ($0<p<1$) diffusion; The case of reaction-diffusion equation $m>1,p=1,b > 0$ is analyzed in \cite{Kalashnikov4,KPV,Abdullaev1}. Surprisingly, only a partial result is available for the double-degenerate PDE \eqref{OP1}. 
It follows from \cite{Ishige} that there exists a weak solution to the CP \eqref{OP1}, \eqref{IF3} for any $\alpha >0$. 
Uniqueness of the solution is an open problem. For our purposes it is satisfactory to employ the notion of the minimal solution.
\begin{definition}[Minimal Solution]\label{def: minimal soln} A nonnegative weak solution $u$ of the CP  \eqref{OP1}, \eqref{IF1} is called a {\it minimal solution} if
\begin{equation}\label{minimalsolution}
0\leq u(x,t)\leq v(x,t),
\end{equation}
for any nonnegative weak solution $v$ of the same problem  \eqref{OP1}, \eqref{IF1}.
\end{definition}
Note that the minimal solution is unique by definition. The following standard comparison result is true in the class of minimal solutions.
\begin{lemma}\label{CT2} Let $u$ and $v$ be minimal solutions of the CP  \eqref{OP1}, \eqref{IF1}. If
\begin{equation*}
u(x,0)\geq (\leq) \ v(x,0), \ x\in \mathbb{R},
\end{equation*}
then
\begin{equation*}
u(x,t)\geq (\leq) \ v(x,t), \ (x,t)\in \mathbb{R}\times (0,T).
\end{equation*}
\end{lemma}
We now establish a series of lemmas that describe preliminary estimations for the CP. The proof of these results is based on nonlinear scaling.  
\begin{lemma}\label{Lemma 1}
If $b=0$ and $\alpha > 0$, then the minimal solution $u$ of the CP \eqref{OP1}, \eqref{IF3} has the self-similar form \eqref{sss0},
where the self-similarity function $f$ satisfies \eqref{f1}.
If $u_0$ satisfies \eqref{IF2}, and $u$ is the unique weak solution to CP \eqref{OP1}, \eqref{IF1}, then $u$ satisfies \eqref{sim1}.
\end{lemma}
The proof coincides with that given for Lemma 3 from \cite{AbdullaPrinkey1}. 
\begin{lemma}\label{Lemma 2}
Let $u$ be a weak solution to the CP \eqref{OP1}, \eqref{IF1}, with $u_0$ satisfying the condition \eqref{IF2}. Let one of the following cases be valid \\
\[
\begin{cases}
b>0, \, 0<\beta<mp,\, 0<\alpha<(1+p)/(mp-\beta) & \text{Case 1}, \\ 
b>0, \beta \geq mp, \, \alpha>0 & \text{Case 2}, \\
b<0,  \beta \geq 1, \, \alpha>0 & \text{Case 3}.
\end{cases}
\]
 Then, for any $\rho\in\mathbb{R}$,  $u$ satisfies \eqref{sim1} with the same function $f$ as in Lemma \ref{Lemma 1}.
\end{lemma}
The proof of Cases 1 and 2 coincides with the proof of Lemma 4 from \cite{AbdullaPrinkey1}. Consider the Case 3. From \eqref{IF2} it follows that for $\forall \epsilon>0$ $\exists x_{\epsilon}<0$ such that
\begin{equation}\label{case3proof-1}
(C-\epsilon/2)(-x)_+^\alpha \leq u_0(x) \leq (C-\epsilon/2)(-x)_+^\alpha, \ x\geq x_\epsilon.
\end{equation}
Assume that $u_{\pm \epsilon}$ is a solution of the boundary value problem
\begin{equation}\label{case3proof-2}
Lu_{\pm \epsilon} = 0,  \, |x|<|x_\epsilon|,  0<t\leq \delta,
\end{equation}
\begin{equation}\label{case3proof-3}
u_{\pm \epsilon}(x,0)=(C\pm \epsilon)(-x)_+^\alpha, \  |x|\leq|x_\epsilon|,
\end{equation}
\begin{equation}\label{case3proof-4}
u_{\pm \epsilon}(x_\epsilon,t)=(C\pm \epsilon)(-x_\epsilon)^\alpha, \ u_{\pm \epsilon}(-x_\epsilon,t)=u(-x_\epsilon,t), \ 0\leq t \leq \delta,
\end{equation}
where $\delta=\delta(\epsilon)>0$ is chosen such that
\begin{equation}\label{case3proof-5}
u_\epsilon(x_\epsilon,t)\geq u(x_\epsilon,t), \ u_{-\epsilon}(x_\epsilon,t)\leq u(x_\epsilon,t), \ 0\leq t \leq \delta.
\end{equation}
From the comparison theorem it follows that
\begin{equation}\label{case3proof-6}
u_{-\epsilon}\leq u \leq u_{\epsilon}, \ |x|\leq |x_\epsilon|, 0\leq t \leq \delta.
\end{equation}
Now if we rescale
\begin{equation}\label{case3proof-7}
u^{\pm \epsilon}_{k}(x,t)=k u_{\pm \epsilon}(k^{-1/\alpha}x, k^{(\alpha(mp-1)-(1+p))/\alpha}t), \ k>0,
\end{equation}
then $u^{\pm \epsilon}_{k}$ satisfies the following problem
\begin{equation}\label{case3proof-8}
L_{k}u^{\pm \epsilon}_{k} \equiv (u^{\pm \epsilon}_{k})_t-\Big(|((u^{\pm \epsilon}_{k})^{m})_x|^{p-1}((u^{\pm \epsilon}_{k})^{m})_x\Big)_x+bk^{(\alpha(mp-\beta)-(1+p))/\alpha}(u^{\pm \epsilon}_{k})^{\beta} = 0 \,  \text{ in } D^k_\epsilon,
\end{equation}
\begin{equation}\label{case3proof-9}
u^{\pm \epsilon}_k(x,0)=(C\pm \epsilon)(-x)_+^\alpha, \  |x|\leq k^{1/\alpha}|x_\epsilon|, 
\end{equation}
\begin{equation}\label{case3proof-10}
u^{\pm \epsilon}_k(k^{1/\alpha}x_\epsilon,t)=k(C\pm \epsilon)(-x_\epsilon)^\alpha, \ u^{\pm \epsilon}_{k}(-k^{1/\alpha} x_\epsilon, t) = ku(-x_\epsilon, k^{\frac{\alpha (mp-1) -1-p}{\alpha}} t), 0 \leq t \leq k^{\frac{p+1 + \alpha (1-mp)}{\alpha}} \delta,
\end{equation}
where
\[ D^k_\epsilon=\{(x,t): |x|<k^{1/\alpha}|x_\epsilon|,  0<t\leq k^{(\alpha(1-mp)+(1+p))/\alpha}\delta\}. \]
The next step is to prove the convergence of the sequence $\{u^{\pm \epsilon}_{k} \}$ as $k\to+\infty$. Consider a function
\[ g(x,t)=(C+1)(1+|x|^\mu)^{\frac{\alpha}{\mu}}(1-\nu t)^\gamma, \ x\in \mathbb{R}, 0\leq t \leq t_0=\nu^{-1}/2, \]
where
\begin{gather} 
\gamma <0, \ \mu>\frac{p+1}{p}, \ \nu=-h_*+1, \ h_*=\min_{\mathbb{R}} h(x) > -\infty\nonumber\\
h(x)=p(\alpha m)^p(C+1)^{p-1}\gamma^{-1} (1-\nu t)^{\gamma(mp-1)+1} (1+|x|^\mu)^{\frac{\alpha(mp-1)-\mu(p+1)}{\mu}} |x|^{(\mu-1)p-1}\nonumber\\
\times[(\mu-1)(1+|x|^\mu)+(\alpha m-\mu)\mu |x|^\mu] \nonumber
\end{gather}
Then we have
\begin{gather}
L_{k}g = -\gamma(C+1)(1+|x|^\mu)^{\frac{\alpha}{\mu}}(1-\nu t)^{\gamma-1} S \ \text{ in } \ D^k_\epsilon,\nonumber\\
S=\nu+h(x)-b(C+1)^{\beta-1}\gamma^{-1}(1+|x|^\mu)^{\frac{\alpha(\beta-1)}{\mu}}(1-\nu t)^{\gamma(\beta-1)+1} k^{\frac{\alpha(mp-\beta)-(1+p)}{\alpha}},\nonumber
\end{gather}
and therefore
\begin{equation}\label{case3proof-11}
S\geq 1+R, \ \text{ in } \ D^k_{0\epsilon}=D^k_\epsilon\cap \{(x,t): 0<t \leq t_0\},
\end{equation}
where 
\[ R=O(k^{mp-1-(1+p)/\alpha}) \ \text{ uniformly for } \ (x,t)\in D^k_{0\epsilon} \ \text{ as } k\to +\infty.  \]
Hence, we have for $0<\epsilon <<1$ and $k>>1$
\begin{equation}\label{case3proof-12}
L_{k}g\geq 0, \ \text{ in } \ D^k_{0\epsilon},
\end{equation}
\begin{equation}\label{case3proof-13}
g(x,0)\geq u^{\pm \epsilon}_k(x,0), \ \text{ for } \ |x|\leq k^{1/\alpha}|x_\epsilon|,
\end{equation}
\begin{equation}\label{case3proof-14}
g(\pm k^{1/\alpha}x_\epsilon,t) \geq u^{\pm \epsilon}_k(\pm k^{1/\alpha}x_\epsilon,t), \ \text{ for } \  0\leq t \leq t_0.
\end{equation}
From \eqref{case3proof-12}-\eqref{case3proof-14} and comparison theorem it follows that
\begin{equation}\label{case3proof-15}
0\leq u^{\pm \epsilon}_k(x,t) \leq g(x,t), \ \text{ in } \ \bar{D}^k_{0\epsilon}
\end{equation}
Let $G$ be an arbitrary fixed compact subset of 
\[ P=\{(x,t): x\in \mathbb{R}, 0<t\leq t_0\} \]
By choosing $k$ to be so large that $G\subset P$, it follows from \eqref{case3proof-15} that the sequences $\{u^{\pm \epsilon}_{k}\}$ are uniformly bounded in G.
From \cite{Ivanov1,Ivanov2}, it follows that they are uniformly H\"{o}lder continuous in G. From the Arzela-Ascoli theorem and standard diagonalization argument it follows that 
there exist functions $v_{\pm\epsilon}$ such that for some subsequence $k'$ 
\begin{equation}\label{case3proof-16}
\lim_{k'\to+\infty} u^{\pm \epsilon}_{k'}(x,t)=v_{\pm\epsilon}(x,t), \ (x,t)\in P.
\end{equation}
It may be easily checked that $v_{\pm\epsilon}$ is a solution of the CP  \eqref{OP1}, \eqref{IF1} with $u_0=(C\pm\epsilon)(-x)_+^\alpha$.
The remainder of the proof coincides with the proof of Lemma 4 from \cite{AbdullaPrinkey1}.
\begin{lemma}\label{Lemma 3}
If $b>0,\, 0<\beta<mp<1$, and $\alpha=(1+p)/(mp-\beta)$, then the minimal solution $u$ to the CP \eqref{OP1}, \eqref{IF2} has the self-similar form \eqref{sss2},
where the self-similarity function $f_1$ satisfies 
\begin{equation}\begin{cases}\label{selfsimilarNODE3}
\mathcal{L}^0f_1 \equiv   \big(|(f_1^{m})'|^{p-1}(f_1^{m})'\big)'+\frac{mp-\beta}{(1+p)(1-\beta)}\zeta f_1'-\frac{1}{1-\beta}f_1-bf^{\beta}_1=0,~\zeta\in\mathbb{R}, \\
f_1(\zeta)\sim C(-\zeta)^{(1+p)/(mp-\beta)}, \text{ as } \zeta 
\downarrow -\infty, \text{ and } f_1(\zeta) \rightarrow 0, \text{ as } \zeta \uparrow +\infty.
\end{cases}\end{equation}
There exists $\ell_1, \, \lambda >0$ such that for any $\ell \in (-\infty, -\ell_1)$ we have
\begin{equation}
u\Big(\ell t^{\frac{mp - \beta}{(1+p)(1 - \beta)}}, t\Big) = \lambda t^{\frac{1}{1-\beta}}, \, t \geq 0.
\end{equation}
If $0<C<C_*$, then we have
\begin{equation}\label{lambda1}
0 < \lambda < C_* (-\ell)^{\frac{1+p}{mp - \beta}},
\end{equation}
while if $C>C_*$, then $f_1(0)=A_1(m,p,\beta,C,b)=A_{1}>0$.
\end{lemma}
\begin{proof}[Proof of Lemma \ref{Lemma 3}]
The first assertion of the lemma is known when $mp\geq1$ (see Lemma 6 of \cite{AbdullaPrinkey1}). 
We define
\begin{equation}\label{scale2}
u_{k}(x,t) = ku(k^{\frac{\beta - mp}{1+p}} x, k^{\beta -1} t), \, k>0.
\end{equation}
It's easy to see that \eqref{scale2} satisfies the CP \eqref{OP1}, \eqref{IF3}. We consider $u$ to be a unique minimal solution of CP \eqref{OP1}, \eqref{IF3} such that 
\begin{equation}\label{U12}
u(x,t) \leq  ku(k^{\frac{\beta - mp}{1+p}} x, k^{\beta -1} t), \, k>0.
\end{equation}
By changing the variables in \eqref{U12} as
\begin{equation}\label{changeofvar}
y= k^{\frac{\beta - mp}{1+p}} x, \, \tau=k^{\beta -1} t,
\end{equation}
we derive \eqref{U12} with opposite inequality and with $k$ replaced with $k^{-1}$. Since $k>0$ is arbitrary, \eqref{U12} follows with "=".
Taking $k=t^{1/(1-\beta)}$, \eqref{scale2} implies \eqref{sss2} with $f_1(\zeta)=u(\zeta,1)$.\\
To prove the second part of the lemma, take arbitrary $x_1<0$. 
Since $u$ is continuous, there exists a number $\delta_1>0$ such that
\begin{equation}\label{est1l3}
\frac{C}{2} (-x_1)^\frac{1+p}{mp - \beta} \leq u(x_1, \delta), \, 0 < \delta \leq \delta_1.
\end{equation}
Furthermore, if $0 < C < C_*$ we have
\begin{equation}\label{est2l3}
u(x_1, \delta) \leq C_* (-x_1)^\frac{1+p}{mp - \beta}, \, 0 < \delta \leq \delta_1.
\end{equation}
Taking $k = t^\frac{1}{1 - \beta} \delta^\frac{1}{\beta -1}$ with
\begin{equation} \nonumber
x = \ell t^\frac{mp - \beta}{(1+p)(1 - \beta)}, \ell = \ell(\delta) = x_1 \delta^\frac{\beta - mp}{(1+p)(1 - \beta)}, \, 0 < \delta \leq \delta_1,
\end{equation}
\eqref{e2l3} follows with
$\lambda = \lambda (\delta) = \delta^\frac{1}{\beta -1} u(x_1, \delta)$ and $\ell_1 = -x_1 \delta_1^\frac{\beta - mp}{(1+p)(1 - \beta)}, \, 0 < \delta \leq \delta_1$. \\
If $0 < C < C_*$, then \eqref{lambda1} follows from \eqref{est1l3}. Let $C>C_*$. To prove that $f_1(0)>0$ it is enough to show that there exists $t_0>0$ such that
\begin{equation}\label{e3l3}
u(0,t_0)>0.
\end{equation}
To do this, we construct a nonnegative subsolution $g(x,t)$ to \eqref{OP1}. Consider
\begin{equation}\nonumber
g(x,t) = C_1 (-x + t)^\frac{1+p}{mp - \beta}_+, \text{ where } C_* < C_1 < C.
\end{equation}
For $x \geq t, \, g(x,t)$ is identically zero and so we automatically have $Lg \leq 0$. For $x<t$, we have
\begin{equation}\nonumber
L g = b C_1^\beta (-x +t)^\frac{\beta (1+p)}{mp - \beta} S,
\end{equation}
where
\begin{equation}\nonumber
S = 1 + C_1^{1-\beta} \frac{1+p}{b(mp - \beta)} (-x +t)^\frac{p(1-m - \beta) +1}{mp - \beta} - \left( \frac{C_1}{C_*} \right)^{mp - \beta}.
\end{equation}
Choosing $x_1 <0, \, t_2 >0$ we have
\begin{equation}\nonumber
Lg \leq 0, \forall \, x_1 \leq x \leq t, \, 0 \leq t \leq t_2.
\end{equation}
We have for $t=0$
\begin{equation}\nonumber
g(x,0) = C_1(-x)^\frac{1+p}{mp - \beta}_+ < C(-x)^\frac{1+p}{mp - \beta} = u_0(x), \, x \geq x_1.
\end{equation}
For $x=x_1$, we have
\begin{equation}\nonumber
g(x_1, 0) = C_1 (-x_1)^\frac{1+p}{mp -\beta} < C(-x_1)^\frac{1+p}{mp - \beta} = u(x_1,0).
\end{equation}
By continuity, there exists a number $\delta > 0$ such that for any $0 \leq t \leq \delta$ we have
\begin{equation}\nonumber
g(x_1, t) \leq u(x_1, t).
\end{equation}
Now letting $t_1 = \text {min} \{ \delta, t_2 \}$ we have
\begin{equation}
0 < g(x,t) \leq u(x,t), \, x_1 \leq x < t, \, 0 \leq t \leq t_1.
\end{equation}
In particular, this gives us $u(0,t_0) > 0, \, 0 \leq t_1 \leq t_1$, hence, $f_1(0) >0$. Lemma \ref{Lemma 3} is proved.
\end{proof}
\begin{lemma}\label{Lemma 4}
Let $b>0,\, 0<\beta<mp<1$, and $\alpha=(1+p)/(mp-\beta)$, and let $u$ be the minimal solution to the CP \eqref{OP1}, \eqref{IF2}. Then $u$ satisfies
\begin{equation}\label{usim4}
u\Big(\ell t^{\frac{mp - \beta}{(1+p)(1 - \beta)}}, t\Big) \sim \lambda t^{\frac{1}{1-\beta}}, \text{ as } t \to 0^+,
\end{equation}
where $\ell_1, \lambda > 0$ are the same as in Lemma \ref{Lemma 3}. 
Furthermore, if $0<C<C_*$, then
\begin{equation}
0 < \lambda < C_* (-\ell)^{\frac{1+p}{mp - \beta}}.
\end{equation}
If $C>C_*$, then
\begin{equation}
u(0,t) \sim A_1 t^\frac{1}{1 -\beta}, \text{ as } t \to 0^+; \, f_1(0) = A_1 >0.
\end{equation}
\end{lemma}
The proof of Lemma \ref{Lemma 4} follows as a localization of the proof of Lemma \ref{Lemma 3}, exactly as local results were proven for Lemma 4 of \cite{AbdullaPrinkey1}.
\begin{lemma}\label{Lemma 5}
If $b>0, \, 0<\beta<mp<1, \text{ and } \alpha>(1+p)/(mp-\beta)$, then the unique weak solution $u$ to the CP \eqref{OP1}, \eqref{IF2} satisfies \eqref{usim3}.
\end{lemma}
The proof of Lemma \ref{Lemma 5} coincides with the proof of Lemma 7 of \cite{AbdullaPrinkey1}. 
\section{Proof of the main results}\label{sec4}

In this section we prove the main results described in Section \ref{sec2}.\\

\begin{proof}[Proof of Theorem \ref{theorem 1}]
From Lemma \ref{Lemma 2}, the asymptotic formulas \eqref{sim1}, \eqref{f1} follow. For arbitrary sufficiently small $\epsilon>0$, from \eqref{sim1}, there exists a number $\delta_1=\delta_{1}(\epsilon)>0$ such that
\begin{equation}\label{fpr1}
(A_0-\epsilon)t^{\alpha/(1+p-\alpha(mp-1))}\leq u(0,t) \leq (A_0+\epsilon)t^{\alpha/(1+p-\alpha(mp-1))}, \, 0\leq t \leq \delta_{1}(\epsilon),
\end{equation}
where $A_0 = f(0)>0$. Consider a function 
\begin{equation}\label{g1}
g(x,t) = t^\frac{1}{1-\beta} f_1(\zeta), \, \zeta = xt^\frac{\beta - mp}{(1+p)(1-\beta)}.
\end{equation}
\begin{equation}\label{lg1}
Lg = t^\frac{\beta}{1-\beta} \mathcal{L}_0 f_1,
\end{equation}
with 
\begin{equation}\label{lg2}
\mathcal{L}_0 f_1 = \frac{\beta - mp}{(1+p)(1-\beta)} \zeta f'_1 + \frac{1}{1-\beta} f_1 - (|(f^m_1)'|^{p-1} (f^m_1)')' + bf_1^\beta.
\end{equation}
We choose
\begin{equation}
f_1 = C_0 (\zeta_0 - \zeta)^\frac{1+p}{mp -\beta}_+, \, 0 < \zeta < +\infty
\end{equation}
with $C_0, \zeta_0 >0$ to be determined. From \eqref{lg2} we have
\begin{equation}\label{lg3}
\mathcal{L}_0 f_1 = b C_0^\beta (\zeta_0 - \zeta)^\frac{(1+p) \beta}{mp - \beta}_+ \left[ 1 - \left( \frac{C_0}{C_*} \right)^{mp - \beta} + \frac{C_0^{1-\beta}}{b(1-\beta)} \zeta_0 (\zeta_0 - \zeta)^\frac{1+p(1-m-\beta)}{mp - \beta}_+ \right].
\end{equation}
Taking $C_0=C_1$ and $\zeta_0=\zeta_1$ (see Appendix, \ref{appendix}) we have
\begin{equation}
\mathcal{L}_0 f_1 \leq b C_1^\beta (\zeta_1 - \zeta)^\frac{(1+p) \beta}{mp - \beta}_+ \left[ 1 - \left( \frac{C_1}{C_*} \right)^{mp - \beta} + \frac{C_1^{1-\beta}}{b(1-\beta)} \zeta_1^\frac{(1+p)(1-\beta)}{mp - \beta} \right] = 0.   \label{l0f1}
\end{equation}
From \eqref{lg1} it follows that
\begin{gather}
Lg = t^\frac{1}{1-\beta} \mathcal{L}_0 f_1 \leq 0, \text{ for } 0 < x < \zeta_1 t^\frac{mp-\beta}{(1+p)(1-\beta)}, \, t>0, \label{compthm1} \\
Lg = 0, \text{ for } x \geq \zeta_1 t^\frac{mp -\beta}{(1+p)(1-\beta)}, \, t>0. \label{compthm2}
\end{gather}
 
Lemma \ref{CT} implies that $g$ is a subsolution of \eqref{OP1} for $x, \, t>0$. 
Since $1/(1-\beta) > \alpha/(1+p - \alpha (mp -1 ))$, there is a number $\delta_2 > 0$ such that
\begin{gather}
g(0,t) \leq u(0,t), \, 0 \leq t \leq \delta_2. \label{compthm3}
\end{gather}
Clearly we have $g(x,0) = u(x,0) = 0, \, x \geq 0$. Fixing $\epsilon=\epsilon_0$ and taking $\delta=(\delta_1, \delta_2)$, together with \eqref{compthm1}, \eqref{compthm2}, and \eqref{compthm3}, the left-hand sides of \eqref{result1}, \eqref{result2} follow. \\

To prove an upper estimation, we first establish a rough upper estimation for the solution 
\begin{equation} \label{upper bound}
u(x,t) \leq D t^\frac{1}{1-mp} x^\frac{1+p}{mp-1}, \, 0 < x < +\infty, \, 0 < t< +\infty.
\end{equation}
This estimation follows directly from Lemma \ref{CT} since the right-hand side of \eqref{upper bound} is a solution to \eqref{OP1} with $b=0$. Using \eqref{upper bound} we can now establish a more accurate estimation, \eqref{result1}. 
Define the region
\[ G_{\ell,\delta} = \{ (x,t) : \zeta_\ell (t) = \ell t^\frac{mp-\beta}{(1+p)(1-\beta)} < x < +\infty, 0 < t \leq \delta \}. \]
We consider a function $g$ of the same form as earlier, with $C_0 = C_*$ and $\zeta=\zeta_2$ in $G_{\ell_0,\delta}$ for some $\ell_0>0$.
From \eqref{lg1} and \eqref{lg3} it follows that
\begin{gather} 
Lg = t^\frac{1}{1-\beta} \mathcal{L}_0 f_1 \geq 0, \text{ for } 0<  x < \zeta_2 t^\frac{mp-\beta}{(1+p)(1-\beta)}, \, t>0, \label{compthm4}\\
Lg = 0, \text{ for } x \geq \zeta_2 t^\frac{mp-\beta}{(1+p)(1-\beta)}, \, t>0. \label{compthm5}
\end{gather}
Taking $\delta>0$ it follows that, from \eqref{upper bound}, we have
\begin{equation}\label{fix1}
u(\zeta_{\ell_0} (t), t) \leq D t^\frac{1}{1-\beta} \ell_0^\frac{1+p}{mp-1} = t^\frac{1}{1-\beta} C_* (\zeta_2 - \ell_0)^\frac{1+p}{mp-\beta} = g(\zeta_{\ell_0} (t), t), \text{ for } 0 \leq t \leq \delta.
\end{equation}
Applying Lemma \ref{CT} in $G_{\ell_0,\delta}$, the right-hand sides of \eqref{result1}, \eqref{result2} follow from \eqref{compthm4}, \eqref{compthm5}, and \eqref{fix1}, which proves the result.
\end{proof}
\begin{proof}[Proof of  Theorem \ref{theorem 2}]
Assume that $u_0$ is defined as \eqref{IF3}. The self-similar solution \eqref{sss2} follows from Lemma \ref{Lemma 3}. The proof of estimation \eqref{result3} when $C>C_*$ (also when $u_0$ is given through \eqref{IF2}) coincides with the proof given in \cite{AbdullaPrinkey1}. 
Let $0<C<C_*$. The formula \eqref{e2l3} follows from Lemma \ref{Lemma 3}. The proof of the right-hand side of \eqref{result3} (also when $u_0$ is given through \eqref{IF2}) coincides with the proof given in \cite{AbdullaPrinkey1}. To prove the left-hand side of \eqref{result3}, consider the function $g$ from \eqref{g1} with
\[f_1(\zeta) =  C_* (-\zeta_5 - \zeta)^\frac{1+p}{mp-\beta}_+, \, \zeta \in \mathbb{R}, \]
and so (see \eqref{lg1}, \eqref{lg2})
\begin{equation}\label{lg4}
Lg \leq 0, \text{ in } G_{-\ell_1, +\infty}.
\end{equation}
Moreover, we have
\begin{equation}
u(-\ell_1 t^\frac{mp-\beta}{(1+p)(1-\beta)}, t) = \lambda t^\frac{1}{1-\beta} = g(-\ell_1 t^\frac{mp-\beta}{(1+p)(1-\beta)}, t) =  t^\frac{1}{1-\beta}C_* ( \ell_1-\zeta_5)^\frac{1+p}{mp-\beta}_+, \label{ul1}
\end{equation}
\begin{equation}
u(x,0)=g(x,0)=0, \, 0\leq x \leq x_0, \label{ul2}
\end{equation}
\begin{equation}
u(x_0,t)=g(x_0,t)=0, \, t\geq0 \label{ul3},
\end{equation}
where $x_0>0$ is an arbitrary fixed number. Applying Lemma \ref{CT} in 
\[G'_{-\ell_1, +\infty} = G_{-\ell_1, +\infty} \cap \{x<x_0\},\]
the lower estimation from \eqref{result3} follows.
Now suppose $u_0$ satisfies \eqref{IF2}. From \eqref{usim4}, it follows that for arbitrary $\epsilon >0$, there exists a number $\delta = \delta (\epsilon) >0$ such that
\[ (\lambda - \epsilon) t^\frac{1}{1-\beta} \leq u(-\ell_1 t^\frac{mp-\beta}{(1+p)(1-\beta)}, t) \leq (\lambda + \epsilon)t^\frac{1}{1-\beta}, \, 0 \leq t \leq \delta. \]
Using this estimation, the left-hand side of \eqref{result3} may be established locally in time. The proof follows as in the global case given above, except that $\lambda$ should be replaced with $\lambda\pm \epsilon$. As in \cite{AbdullaPrinkey1}, \eqref{sss2} and \eqref{result3} imply \eqref{eta2} and \eqref{result4}.
\end{proof}
\begin{proof}[Proof of Theorem \ref{theorem 3}]
The asymptotic estimate \eqref{usim3} follows from Lemma \ref{Lemma 4}. The proof of the asymptotic estimate \eqref{etasim3} coincides with the proof given in \cite{AbdullaPrinkey1}. 
\end{proof}
\begin{proof}[Proof of Theorem \ref{theorem 4}]
The asymptotic estimation \eqref{sim1} is proved in Lemma \ref{Lemma 2}. From \eqref{sim1}, \eqref{fpr1} follows. 
The function
\begin{equation}
g(x,t) = t^\frac{1}{1-mp} \phi (x).
\end{equation}
is a solution of \eqref{OP1}.
Since $1/(1-mp)>\alpha(1+p+\alpha(1-mp))^{-1}$, there exists $\delta>0$ such that 
 \[u(0,t)\geq (A_0-\epsilon)t^{\frac{\alpha}{1+p+\alpha (1-mp)}}\geq t^{\frac{1}{1-mp}}=g(0,t),\qquad 0\leq t \leq \delta.\]
 \[u(x,0)=g(x,0)=0,\qquad 0\leq x<\infty\]
 Therefore, from Lemma~\ref{CT}, the left-hand side of \eqref{MR41} follows. Let us prove the right-hand side of \eqref{MR41}. 
For all $\epsilon>0$ and consider a function 
  \[g_\epsilon(x,t)=(t+\epsilon)^{1/(1-mp)}\phi(x),\]
   \[g_\epsilon(0,t)=(t+\epsilon)^{1/(1-mp)}\phi(0)=(t+\epsilon)^{1/(1-mp)}\geq\epsilon^{1/(1-mp)}\geq\]
   \[\geq \big(A_0+\epsilon\big) t^{\frac{\alpha}{1+p+\alpha (1-mp)}}\geq u(0,t),\;\;\text{for}\;0\leq t\leq \delta_\epsilon= \big[(A_0+\epsilon\big)^{-1}\epsilon^{1/(1-mp)}\big]^{\frac{1+p+\alpha (1-mp)}{\alpha}},\]
  From the Lemma~\ref{CT}, the right-hand side of \eqref{MR41} follows.
  
Intergration of \eqref{phiODE1} implies \eqref{ODEsolution1}.
By rescaling $x\rightarrow \epsilon^{-1}x, \epsilon>0$ from \eqref{ODEsolution1} we have
\begin{equation*}
\frac{x}{\epsilon}= \displaystyle\int_{\phi(\frac{x}{\epsilon})}^{1} \frac{m}{s} \left[ \frac{b}{p} + \frac{m(1+p)}{p(1-mp)(1+m)} s^{1-mp} \right]^{-\frac{1}{1+p}}ds.
\end{equation*}
Change of variable $z=-\epsilon \ln y$ implies
\begin{equation}\label{rescaledsolution}
x= {\mathcal F}[\Lambda_{\epsilon}(x)],
\end{equation}
where 
\[ {\mathcal F}(y)=\int_{0}^{y} m \left[ \frac{b}{p} + \frac{m(1+p)}{p(1-mp)(1+m)} e^{z(mp-1)/\epsilon} \right]^{-\frac{1}{1+p}}dz, \]
\[\Lambda_{\epsilon}(x) = -\epsilon\ln\phi \Big(\frac{x}{\epsilon}\Big).\]
From \eqref{rescaledsolution} it follows that
\begin{equation}\label{rescaledsolution1}
\Lambda_{\epsilon}(x)]={\mathcal F}^{-1}(x),
\end{equation}
where ${\mathcal F}^{-1}$ is an inverse function of ${\mathcal F}$.
Since $0<mp<1$ it easily follows that
\begin{equation}\label{rescaledsolution2}
\lim_{\epsilon\to 0}{\mathcal F}(y)=m(b/p)^{-\frac{1}{1+p}}y, \ \lim_{\epsilon\to 0}{\mathcal F}^{-1}(y)=m^{-1}(b/p)^{\frac{1}{1+p}}y,
\end{equation}
for $y\geq 0$, and the convergence is uniform in bounded subsets of $\mathbb{R}^+$. From
\eqref{rescaledsolution1}, \eqref{rescaledsolution2} it follows that
\begin{equation}\label{rescaledsolution3}
-\lim_{\epsilon\to0^+}\epsilon\ln\phi\left(\frac{x}{\epsilon}\right) = m^{-1}(b/p)^{\frac{1}{1+p}}x.
\end{equation}
By letting $y=x/\epsilon$, the estimate \eqref{phi41} follows. Estimation \eqref{phiest1}, and accordingly also \eqref{philim2},\eqref{uest1},\eqref{ulim2} easily follow from \eqref{ODEsolution1}, \eqref{ODEsolution2}.
\end{proof}
\begin{proof}[Proof of Theorem \ref{theorem 5}]
Let either $b>0, \, \beta > mp$ or $b<0, \beta \geq 1$. The asymptotic estimation \eqref{sim1} follows from Lemma \ref{Lemma 2}. Take arbitrary small $\epsilon>0$. From \eqref{sim1}, there exists a number $\delta_1=\delta_1(\epsilon)>0$ such that \eqref{fpr1} holds. Let $\beta\geq1$, and consider a function
\begin{equation}\label{g51}
g(x,t) = t^\frac{\alpha}{1+p+\alpha(1-mp)} f(\xi), \, \xi = xt^\frac{-1}{1+p+\alpha (1-mp)}.
\end{equation}
We have
\begin{equation}\label{lg51}
Lg = t^\frac{\alpha mp - 1 -p}{1+p +\alpha (1-mp)} L_1 f,
\end{equation}
where
\begin{equation}\label{lg52}
\begin{aligned}
L_1 f = \frac{\alpha}{1+p+\alpha (1-mp)} f - \frac{1}{1+p + \alpha (1-mp)} \xi f'- \\
- (|(f^m)'|^{p-1} (f^m)')' + bt^\frac{1+p-\alpha (mp-\beta)}{1+p+\alpha(1-mp)} f^\beta.
\end{aligned}
\end{equation}
As a function $f$ we select
\begin{equation}\label{f5}
f(\xi) = C_0 (\xi_0 + \xi)^\frac{1+p}{mp-1}, \, \xi \geq 0,
\end{equation}
where $C_0$ and $\xi_0$ are positive constants. From \eqref{lg52} we have
\begin{equation}\label{lg5f}
\begin{aligned}
L_1 f = \frac{1}{1+p+\alpha (1-mp)} C_0 (\xi_0 + \xi)^\frac{1+p}{mp-1}\cross\\
\cross \left[ R(\xi) + b(1+p+\alpha(1-mp))t^\frac{1+p-\alpha(mp-\beta)}{1+p+\alpha(1-mp)} C_0^{\beta -1} (\xi_0 + \xi)^\frac{(1+p)(\beta-1)}{mp-1} \right],
\end{aligned}
\end{equation}
where
\begin{equation}\label{R51}
R(\xi) = \alpha - \frac{(1+p+\alpha(1-mp))(m(1+p))^p p(m+1)}{(1-mp)^{p+1}} C_0^{mp-1} + \frac{1+p}{1-mp} \xi (\xi_0+\xi)^{-1}.
\end{equation}
To prove an upper estimation we take $C_0=C_6$ and $\xi_0=\xi_2$ (see Appendix, \ref{appendix}). Then we have
\begin{equation}
R(\xi) \geq  \alpha \frac{\mu_b -1}{\mu_b}, \text{ for } \xi\geq0 \label{RC0}.
\end{equation}
From \eqref{lg5f} it follows that
\[L_1 f \geq 0, \text{ for } \xi \geq 0, \, 0\leq t\leq \delta_2, \]
where 
\[\delta_2=\delta_1, \, \text{ if } b>0; \, \delta_2=\min(\delta_1, \delta_3), \, \text{ if } b<0,\]
and 
\[\delta_3 = \left[ \frac{ \alpha \epsilon (A_0 + \epsilon)^{1-\beta}}{(1+\epsilon) (-b(1+p+\alpha (1-mp)))} \right]^\frac{1+p+\alpha (1-mp)}{1+p  + \alpha (\beta - mp)}.\]
Hence, from \eqref{lg51} we have
\begin{equation}\label{lg53}
Lg \geq 0, \text{ for } 0\leq x<+\infty,\, 0\leq t \leq\delta_2. 
\end{equation}
From \eqref{fpr1} and Lemma \ref{CT}, the right-hand side of \eqref{result5} follows with $\delta=\delta_2$. To prove a lower bound in this case we take $C_0=C_5$ and $\xi_0=\xi_1$. If $b>0$ and $\beta < \frac{p(1-m)+2}{1+p}$, from \eqref{lg5f} we have
\begin{equation}\label{R52}
\begin{aligned}
R(\xi) \leq \alpha - \frac{(1+p+\alpha(1-mp))(m(1+p))^p p(m+1)}{(1-mp)^{p+1}} C_5^{mp-1} + \frac{1+p}{1-mp} =\\ =\frac{\epsilon}{\epsilon -1} \left(\alpha + \frac{1+p}{1-mp} \right),
\end{aligned}
\end{equation}
and so
\begin{equation}\label{lg54}
L_1 f \leq 0, \text{ for } 0\leq x<+\infty,\, 0\leq t \leq \delta_4,
\end{equation}
where $\delta_4 = \min(\delta_1, \delta_5)$ and
\[\delta_5 = \left[ \frac{(A_0 - \epsilon)^{1-\beta} \epsilon}{b(1-mp)(1-\epsilon)} \right]^\frac{1+p+\alpha (1-mp)}{1+p - \alpha (mp-\beta)}.\]
From \eqref{lg54} it follows that
\begin{equation}\label{lg55}
Lg \leq 0, \text{ for } 0\leq x<+\infty,\, 0\leq t \leq \delta_4. 
\end{equation}
If either $b<0, \, \beta \geq 1$ or $b>0, \, \beta \geq \frac{p(1-m)+2}{1+p}$, from \eqref{lg5f} we have
\begin{multline}\label{lg56}
L_1 f = \frac{1}{1+p+\alpha (1-mp)} C_5 (\xi_1 + \xi)^\frac{1+p+1-mp}{mp-1} \cross\\
\cross \Big[ R_1 (\xi) + b(1+p+\alpha (1-mp)t^\frac{1+p-\alpha (mp-\beta)}{1+p + \alpha (1-mp)} C_5^{\beta -1} (\xi_1 + \xi)^\frac{1-mp+(1+p)(1-\beta)}{1-mp} \Big],
\end{multline}
where 
\begin{equation}\label{R53}
\begin{aligned}
R_1 (\xi) = \alpha (\xi_1 +\xi) - \frac{(1+p+\alpha (1-mp))(m(1+p))^p p(m+1)}{(1-mp)^{p+1}} C_5^{mp-1} (\xi_1 + \xi) +\\ +\frac{1+p}{1-mp} \xi,
\end{aligned}
\end{equation}
which again imply \eqref{lg54}, where $\delta_4 = \delta_1$ if $b<0$, $\delta_4=\min(\delta_1, \delta_5)$ if $b>0$, where
\[\delta_5 = \left[ \frac{1+p}{(1-mp)b(1+p+\alpha (1-mp))} (A_0 - \epsilon)^{1-\beta} \right]^\frac{1+p+\alpha (1-mp)}{1+p - \alpha (mp-\beta)}.\]
As before \eqref{lg55} follows from \eqref{lg56}. From \eqref{fpr1} and Lemma \ref{CT}, the left-hand side of \eqref{result5} follows with $\delta=\delta_4$. Therefore, \eqref{result5} is proved with $\delta=\min(\delta_2,\delta_4)$. 

Let $b>0$ and $\beta\geq1$. The upper estimation \eqref{diffbound} follows directly from the Lemma \ref{CT}, since the right-hand side is a solution of \eqref{OP1} with $b=0$. Let $\beta \geq \frac{p(1-m)+2}{1+p}$. Fixe $\epsilon=\epsilon_0$ and take $\delta=\delta(\epsilon_0)>0$ in \eqref{result5}. From the left-hand side of \eqref{result5} and \eqref{diffbound}, \eqref{asymp1} follows. However, if $b>0$ and $1\leq\beta <\frac{p(1-m)+2}{1+p}$, from \eqref{result5} and \eqref{diffbound} it follows that for any fixed $t \in (0,\delta(\epsilon)]$ that

\[ D(1-\epsilon)^\frac{1}{1-mp} \leq \liminf_{x \to + \infty} u t^\frac{1}{mp-1}x^\frac{1+p}{1-mp} \leq \limsup_{x \to +\infty} u t^\frac{1}{mp-1}x^\frac{1+p}{1-mp} \leq D. \]
Since $\epsilon>0$ is arbitrary, \eqref{asymp2} follows. 
Letting $b<0$ and $\beta \geq 1$, we prove \eqref{fasymp1}. Consider a function
\[\bar{g} (x,t) = D(1-\epsilon)^\frac{1}{mp-1} t^\frac{1}{1-mp} x^\frac{1+p}{mp-1},\] 
in $G= \{ (x,t) : \mu t^\frac{1}{1+p+\alpha (1-mp)} < x < +\infty, 0 \leq t \leq \delta \}$, where $\mu$ is as defined in \eqref{mu1}.  Let $g(x,t) = \bar{g} (x,t)$ for $(x,t) \in \bar{G}/(0,0)$, and let $g(0,0) = 0$. Then we have
\[Lg = \frac{D}{1-mp} (1-\epsilon)^\frac{1}{mp-1} t^\frac{mp}{1-mp} x^\frac{1+p}{1-mp} H, \text{ in } G, \] where
\[H=\epsilon + bD^{\beta-1} (1-\epsilon)^\frac{\beta-1}{mp-1} (1-mp)t^\frac{\beta - mp}{1-mp} x^\frac{(1+p)(\beta-1)}{mp-1}. \]
We then have
\[ H \geq \epsilon + bD^{\beta-1} (1-\epsilon)^\frac{\beta -1}{mp-1} (1-mp) \mu^\frac{(1+p)(\beta -1)}{mp-1} t^\frac{1+p+\alpha (\beta - mp)}{1+p+\alpha (1-mp)}, \text{ in } G. \]
It follows that 
\[H\geq 0,  \text{ in } G, \text{ for } \delta \in (0, \delta_0],\]
\[\delta_0 = \left[ \frac{\epsilon (1-\epsilon)^\frac{\beta-1}{1-mp} \mu^\frac{(1+p)(\beta-1)}{1-mp}}{b(1-mp)D^{\beta-1}} \right]^\frac{1+p+\alpha (1-mp)}{1+p + \alpha (\beta-mp)}, \]
so we have
\begin{equation}\label{lg56}
Lg \geq 0 \text{ in } G.
\end{equation}
Moreover, we have that
\[ g(\mu t^\frac{1}{1+p+\alpha (1-mp)}, t) = (A_0+\epsilon)t^\frac{\alpha}{1+p+\alpha (1-mp)}, \text{ for } 0\leq t \leq \delta. \]
From \eqref{result5}, it follows that
\[u(\mu t^\frac{1}{1+p+\alpha (1-mp)}, t) \leq (A_0+\epsilon)t^\frac{\alpha}{1+p+\alpha (1-mp)}, \text{ for } 0\leq t \leq \delta.  \]
Therefore, we have that
\begin{equation}\label{estg5}
g \geq u, \text{ on } \bar{G}\setminus G.
\end{equation}
From \eqref{lg56}, \eqref{estg5}, and Lemma \ref{CT}, the desired estimation \eqref{fasymp1} follows. Since $\epsilon > 0$ is arbitrary, from the left-hand side of \eqref{result5} and \eqref{fasymp1}, \eqref{asymp1} follows as before. 
Let $b>0$ and $0<mp<\beta<1$. The left-hand side of \eqref{mpbeta1} may be proved as the left-hand side of \eqref{result1} was previously proved. The only difference being that we take $f_1(\zeta) = C_*(1-\epsilon)(\zeta_8 + \zeta)^\frac{1+p}{mp-\beta}_+$ in \eqref{g1}, \eqref{lg1}. From \eqref{result5}, it follows that for any fixed $t \in (0, \delta]$, where $\delta$ is independent of $\epsilon$, we have
\[C_*(1-\epsilon)\leq \liminf_{x \to + \infty} u x^\frac{1+p}{\beta-mp} \leq \limsup_{x \to +\infty} u x^\frac{1+p}{\beta-mp} \leq C_*.\]
Since $\epsilon>0$ is arbitrary, \eqref{asymp3} follows. 

Now, let $b=0$. First assume that $u_0$ is defined by \eqref{IF3}. The self-similar form \eqref{sss0} and the formula \eqref{f1} follow from Lemma \ref{Lemma 1}. To prove \eqref{ubb0}, again consider the function $g$ as in \eqref{g51}, which satisfies \eqref{lg51} with $b=0$. As a function $f$ take \eqref{f5}. Then we derive \eqref{lg5f} with $b=0$. To prove an upper estimate we take $C_0=C_7$, $\xi_0=\xi_4$ and from \eqref{R51} we have
\[ R(\xi) \geq \alpha - \frac{(1+p+\alpha (1-mp))}{1-mp} \left( \frac{C_7}{D} \right)^{mp-1}=0, \]
which implies \eqref{lg53} with $\delta_2=+\infty$. As before, from \eqref{lg53} and Lemma \ref{CT}, the right-hand side of \eqref{ubb0} follows. The left-hand side of \eqref{ubb0} may be established similarly by choosing $C_0=D$ and $\xi_0=\xi_3$. To prove estimation \eqref{diffbound} consider
\[ g_{\mu} (x,t) = D (t+\mu)^\frac{1}{1-mp} (x+\mu)^\frac{1+p}{mp-1}, \, \mu>0, \]
which is a solution of \eqref{OP1} when $x>0$ and $t>0$. Since 
\[ u(0,t) \leq D \mu^\frac{p}{mp-1} \leq g_\mu (0,t), \text{ for } 0 \leq t \leq T(\mu), \]
where
\[T(\mu) = \left[ \frac{D}{A_0} \mu^\frac{p}{mp-1} \right]^\frac{1+p+\alpha (1-mp)}{\alpha}, \]
Lemma \ref{CT} implies
\[ u(x,t) \leq g_\mu (x,t), \, 0 < x < +\infty, \, 0 \leq t \leq \text{min}(\delta_1, T(\mu)). \]
Letting $\mu \to 0^+$, \eqref{diffbound} can be easily derived. From \eqref{diffbound} and \eqref{ubb0} it follows that for any fixed $t>0$ the formula \eqref{asymp1} is valid. If $u_0$ satisfies \eqref{IF2} with $\alpha>0$, then \eqref{sim1} and \eqref{fpr1} follow from Lemma \ref{Lemma 1}. In a similar way, we can prove that for arbitrary sufficiently small $\epsilon>0$, there exists a number $\delta=\delta(\epsilon)>0$, such that \eqref{ubb0} is valid for $0\leq t \leq \delta(\epsilon)$, however, $A_0$ should be replaced with $A_0-\epsilon$ on the left-hand side and $A_0+\epsilon$ on the right-hand side. From the local analog of \eqref{ubb0} and \eqref{diffbound}, for any fixed $t\in (0, \delta]$, the formula \eqref{asymp1} is valid.
\end{proof}
\section{Conclusions}\label{sec5}

This paper presents a full classification of the short-time behavior of the interfaces and local solutions near the interfaces or at infinity for the Cauchy problem for the nonlinear double degenerate type reaction-diffusion equation of turbulent filtration in the case of fast diffusion
\[ u_t=(|(u^{m})_x|^{p-1}(u^{m})_x)_x-bu^{\beta}, \, x\in \mathbb{R}, \, 0<t<T, \, 0<mp<1, \, \beta >0,\]
with
\[u(x,0)=u_0(x)\sim C(-x)_+^{\alpha},\, \text{as } \, x\rightarrow 0^{-},\, \text{for some }  C>0,\, \alpha>0,\]
and either $b\geq 0$ or $b<0, \, \beta\geq 1$.
The classification is based on the relative strength of the diffusion and absorption forces. The following is a summary of the main results:
\begin{itemize}
\item If $b>0, \, 0<\beta<mp$, and $0<\alpha<(1+p)/(mp-\beta)$, then diffusion weakly dominates over the absorption and the interface expands with asymptotic formula given by
\[\eta(t)\sim \psi(C, m, p,\alpha)t^{(mp-\beta)/(1+p)(1-\beta)}, \text{ as } \, t\to 0^+, \]
where, $\psi(C, m, p, \alpha)>0$.
\item If $b>0, \, 0<\beta <mp$, and $\alpha=(1+p)/(mp-\beta)$, then diffusion and absorption are in balance, and there is a critical value $C_*$ such that the interface expands or shrinks accordingly as $C>C_*$ or $C<C_*$ and
\[\eta(t) \sim\zeta_*(C, m, p) t^{(mp-\beta)/(1+p)(1-\beta)}, \text{ as } \, t\to 0^+,\]
where $\zeta_* \lessgtr 0$ if $C\lessgtr C_*$.
\item If $b>0, \, 0<\beta<mp$, and $\alpha >(1+p)/(mp-\beta)$, then absorption strongly dominates over diffusion and the interface shrinks with asymptotic formula given by
\[
\eta(t)\sim -\ell_*(C, m, p, \alpha, \beta) t^{1/\alpha(1-\beta)}, \text{ as } \, t\rightarrow 0^+,
\]
where, $\ell_*(C, m, p, \alpha, \beta)>0$.
\item If $b>0, \, 0<\beta = mp <1$, and $\alpha>0$, then domination of the diffusion over absorption is moderate, there is an infinite speed of propagation, and the solution has exponential decay at infinity.
\item If either $b>0, \, \beta>mp$ or $b<0, \, \beta\geq 1$, then diffusion strongly dominates over the absorption, and the solution has power type decay at infinity independent of $\alpha >0$, which coincides with the asymptotics of the fast diffusion equation ($b=0$).
\end{itemize}
\section{Appendix A}\label{appendix}
We give here explicit values of the constants used in Sections \ref{sec2}, \ref{sec2a}, and \ref{sec4}.\\
I. $0<\beta<mp$ and $0<\alpha<(1+p)/(mp-\beta)$\\
\[ \zeta_1 =  (b(1- \beta))^\frac{mp -1}{(1+p)(1-\beta)} (m(1+p))^\frac{p}{1+p} (p(m+\beta))^\frac{1}{1+p}(mp-\beta)^\frac{p(m+\beta -1) -1}{(1+p)(1-\beta)} (1-mp)^\frac{1-mp}{(1+p)(1-\beta)}, \]
\begin{equation*}
C_1 = \left( \frac{1-\beta}{1 - mp} \right)^\frac{1}{mp -\beta} C_*,
\end{equation*}
\begin{equation*}
\ell_0 = \bigg(\frac{1-mp}{mp-\beta}\bigg)^\frac{mp-1}{1-\beta}\bigg(\frac{D}{C_*}\bigg)^\frac{(mp-1)(\beta-mp)}{(1+p)(1-\beta)}, \, \zeta_2 = \ell_0\frac{1-\beta}{mp-\beta}.
\end{equation*}
II. $0<\beta<mp$ and $\alpha=(1+p)/(mp-\beta)$\\
\begin{equation*}C_2 = A_1 \zeta_3^\frac{1+p}{\beta-mp}, \, A_1 = f_1(0),\end{equation*}
\[\zeta_3 =  (m(1+p))^\frac{p}{p+1} (m+\beta)^\frac{1}{1+p} p^\frac{1}{1+p} (mp-\beta)^{-1} (1-\beta)^\frac{1}{1+p} A_1^\frac{m-1}{1+p}\Big[b(1-\beta) A_1^{\beta-1} + 1\Big]^\frac{-1}{1+p},\]
\begin{equation*}
\zeta_4 = \left( \frac{A_1}{C_*} \right)^\frac{mp-\beta}{1+p}, \, \zeta_5 = \ell_1 - \left(\frac{\lambda}{C_*} \right)^\frac{1+p}{mp-\beta}, \, C_3 = C \left( \frac{1}{1-\delta_* \Gamma} \right)^\frac{1+p}{mp-\beta},
\end{equation*}
\begin{equation*}
\Gamma = 1 - \left(\frac{C}{C_*} \right)^\frac{mp-\beta}{1+p}, \, \zeta_6 = \delta_* \Gamma \ell_2, \, \delta_* \text{ satisfies } g(\delta_*) = \max_{\delta \in (0,1)} g(\delta),
\end{equation*}
\begin{equation*}
g(\delta) = (\delta \Gamma)^\frac{1+p-p(m+\beta)}{(1+p)(1-\beta)} \left[ 1 - \delta \Gamma - \left( \frac{C}{C_*} \right)^{mp-\beta} \left( \frac{1}{1-\delta \Gamma} \right)^p \right]^\frac{mp-\beta}{(1+p)(1-\beta)},
\end{equation*}
\begin{equation*}
 \ell_2 = C^\frac{\beta-mp}{1+p} \left[ \frac{b(1-\beta)}{\delta_* \Gamma} \left( 1 - \delta_* \Gamma - \left( \frac{C}{C_*} \right)^{mp-\beta} \left( \frac{1}{1-\delta_* \Gamma} \right)^p \right) \right]^\frac{mp-\beta}{(1+p)(1-\beta)}.
\end{equation*}
V. $\beta>mp$ \\
\begin{small}\[C_5 = (1- \epsilon)^\frac{1}{1-mp} D,\]
\[C_6 = \bigg(\frac{\alpha(1-mp)^{p+1}}{\mu_b(1+p+\alpha(1-mp))(m(1+p))^p p (m+1)}\bigg)^\frac{1}{mp-1},\]
\[\xi_1 = (A_0-\epsilon)^{(mp-1)/(1+p)}(1-\epsilon)^{1/(1+p)}D^{(1-mp)/(1+p)}, \text{ if } b>0, \, 1 \leq \beta < (p(1-m)+2)/(1+p),\]
\[\xi_1 = (A_0-\epsilon)^{(mp-1)/(1+p)}D^{(1-mp)/(1+p)}, \text{ if } b>0, \, \beta \geq (p(1-m)+2)/(1+p) \text{ or } b<0, \, \beta\geq 1,\]
\[\xi_2 = \bigg(\frac{A_0+\epsilon}{C_6}\bigg)^{\frac{mp-1}{1+p}},\]
\[A_0 = f(0)>0,\]
\[\mu_b = \begin{cases}
1,\text{ if }b>0,\\ 
1+\epsilon,\text{  if } b<0,
\end{cases}\]\end{small}
\[ \zeta_8 = \left[ b(1-\beta) C_*^{\beta-1} (1-\epsilon)^{mp-1} (1-(1-\epsilon)^{\beta-mp}) \right]^\frac{mp-\beta}{(1+p)(1-\beta)}, \]
\[\xi_3 = (A_0/D)^{(mp-1)/(1+p)},\]
\[\xi_4 = \xi_3(1+(1+p)/\alpha(1-mp))^{1/(1+p)},\]
\[C_7=D(1+(1+p)/\alpha(1-mp))^{1/(1+mp)}.\]
	
	\FloatBarrier
      	\bibliographystyle{plain}
	\bibliography{references}

\begin{thebibliography}{10}

\bibitem{Abdulla20}
U.~G. Abdulla.
\newblock Local structure of solutions of the {D}irichlet problem for
  {$N$}-dimensional reaction-diffusion equations in bounded domains.
\newblock {\em Adv. Differential Equations}, 4(2):197--224, 1999.

\bibitem{Abdulla5}
U.~G. Abdulla.
\newblock Reaction-diffusion in a closed domain formed by irregular curves.
\newblock {\em Journal of Mathematical Analysis and Applications},
  246:480--492, 2000.

\bibitem{Abdulla4}
U.~G. Abdulla.
\newblock Reaction-diffusion in irregular domains.
\newblock {\em Journal of Differential Equations}, 164(2):321--354, 2000.

\bibitem{Abdulla18}
U.~G. Abdulla.
\newblock On the {D}irichlet problem for reaction-diffusion equations in
  non-smooth domains.
\newblock In {\em Proceedings of the {T}hird {W}orld {C}ongress of {N}onlinear
  {A}nalysts, {P}art 2 ({C}atania, 2000)}, volume~47, pages 765--776, 2001.

\bibitem{Abdulla6}
U.~G. Abdulla.
\newblock On the {D}irichlet problem for the nonlinear diffusion equation in
  non-smooth domains.
\newblock {\em Journal of Mathematical Analysis and Applications},
  260(2):384--403, 2001.

\bibitem{Abdulla3}
U.~G. Abdulla.
\newblock Evolution of interfaces and explicit asymptotics at infinity for the
  fast diffusion equation with absorption.
\newblock {\em Nonlinear Analysis: Theory, Methods, \& Applications},
  50(4):541--560, 2002.

\bibitem{Abdulla16}
U.~G. Abdulla.
\newblock First boundary value problem for the diffusion equation. {I}.
  {I}terated logarithm test for the boundary regularity and solvability.
\newblock {\em SIAM J. Math. Anal.}, 34(6):1422--1434, 2003.

\bibitem{Abdulla14}
U.~G. Abdulla.
\newblock Multidimensional {K}olmogorov-{P}etrovsky test for the boundary
  regularity and irregularity of solutions to the heat equation.
\newblock {\em Bound. Value Probl.}, (2):181--199, 2005.

\bibitem{Abdulla8}
U.~G. Abdulla.
\newblock Reaction-diffusion in nonsmooth and closed domains.
\newblock {\em Boundary Value Problems}, (2):28, 2005.

\bibitem{Abdulla7}
U.~G. Abdulla.
\newblock Well-posedness of the {D}irichlet problem for the nonlinear diffusion
  equation in non-smooth domains.
\newblock {\em Transactions of the American Mathematical Society},
  357(1):247--265, 2005.

\bibitem{Abdulla9}
U.~G. Abdulla.
\newblock Wiener's criterion for the unique solvability of the dirichlet
  problem in arbitrary open sets with non-compact boundaries.
\newblock {\em Nonlinear Analysis: Theory, Methods \& Applications},
  67(2):563--578, 2007.

\bibitem{Abdulla11}
U.~G. Abdulla.
\newblock Wiener's criterion at {$\infty$} for the heat equation.
\newblock {\em Adv. Differential Equations}, 13(5-6):457--488, 2008.

\bibitem{Abdulla12}
U.~G. Abdulla.
\newblock Wiener's criterion at {$\infty$} for the heat equation and its
  measure-theoretical counterpart.
\newblock {\em Electron. Res. Announc. Math. Sci.}, 15:44--51, 2008.

\bibitem{AbdullaPrinkey1}
U.~G. Abdulla, J.~Du, A.~Prinkey, C.~Ondracek, and S.~Parimoo.
\newblock Evolution of interfaces for the nonlinear double degenerate parabolic
  equation of turbulent filtration with absorption.
\newblock {\em Mathematics and Computers in Simulation}, 153:59--82, 2018.

\bibitem{AbdullaJeli1}
U.~G. Abdulla and R.~Jeli.
\newblock Evolution of interfaces for the non-linear parabolic $p$-{L}aplacian
  type reaction-diffusion equations.
\newblock {\em European Journal of Applied Mathematics}, 28(5), 2017.

\bibitem{AbdullaJeli2}
U.~G. Abdulla and R.~Jeli.
\newblock Evolution of interfaces for the non-linear parabolic $p$-{L}aplacian
  type reaction-diffusion equations. ii. fast diffusion vs absorption.
\newblock {\em European Journal of Applied Mathematics}, 2019.

\bibitem{Abdulla1}
U.~G. Abdulla and J.~R. King.
\newblock Interface development and local solutions to reaction-diffusion
  equations.
\newblock {\em SIAM Journal on Mathematical Analysis}, 32(2):235--260, 2000.

\bibitem{Abdullaev1}
U.~G. Abdullaev.
\newblock On existence of unbounded solutions on nonlinear heat equation with
  absorption.
\newblock {\em Zh. Vychisl. Mat. i Mat. Fiz.}, 33:232--245, 1993.

\bibitem{Abdulla25}
U.~G. Abdullaev.
\newblock On sharp local estimates for the support of solutions in problems for
  nonlinear parabolic equations.
\newblock {\em Mat. Sb.}, 186(8):3--24, 1995.

\bibitem{Abdulla21}
U.~G. Abdullaev.
\newblock Instantaneous shrinking of the support of a solution of a nonlinear
  degenerate parabolic equation.
\newblock {\em Mat. Zametki}, 63(3):323--331, 1998.

\bibitem{shmarev}
S.~N. Antontsev, J.~I. Diaz, and S.~Shmarev.
\newblock {\em Energy Methods for Free Boundary Problems: Applications to
  Nonlinear PDEs and Fluid Mechanics}, volume~48.
\newblock Springer Verlag, 2012.

\bibitem{Barenblatt1}
G.~I. Barenblatt.
\newblock On some unsteady motions of a liquid or a gas in a porous medium.
\newblock {\em Prikl. Mat. Mech.}, 16:67--78, 1952.

\bibitem{Barenblatt2}
G.~I Barenblatt.
\newblock {\em Scaling, self-similarity, and intermediate asymptotics}.
\newblock Cambridge Texts in Applied Mathematics. Cambridge University Press,
  1996.

\bibitem{BCP}
P.~Benilan, M.~G. Crandall, and M.~Pierre.
\newblock Solutions of the porous medium equation under optimal conditions on
  initial values.
\newblock {\em Indiana University Mathematics Journal}, 33:51--87, 1984.

\bibitem{Degtyarev}
S.~P. Degtyarev and A.F. Tedeev.
\newblock On the solvability of the cauchy problem with growing initial data
  for a class of anisotropic parabolic equations.
\newblock {\em Journal of Mathematical Sciences}, 181(1):28--46, 2012.

\bibitem{Dibe-Sv}
E.~DiBenedetto.
\newblock {\em Degenerate Parabolic Equations}.
\newblock Series Universitext. Springer Verlag, 1993.

\bibitem{DibeHerrero1}
E.~DiBenedetto and M.~A. Herrero.
\newblock On the {C}auchy problems and initial traces for a degenerate
  parabolic equation.
\newblock {\em Transactions of the American Mathematical Society},
  314:187--224, 1989.

\bibitem{DibeHerrero2}
E.~DiBenedetto and M.~A. Herrero.
\newblock Nonnegative solutions of the evolution $p$-{L}aplacian equations:
  Initial traces and {C}auchy problem when $1<p<2$.
\newblock {\em Archive for Rational Mechanics Analysis}, 111:225--290, 1990.

\bibitem{Esteban}
J.~R. Esteban and J.~L. Vazquez.
\newblock On the equation of turbulent filtration in one-dimensional porous
  media.
\newblock {\em Nonlinear Analysis: Theory, Methods, \& Applications},
  10(11):1303--1325, 1986.

\bibitem{Gala1}
V.~A. Galaktionov, S.~I. Shmarev, and J.~L. Vazquez.
\newblock Regularity of interfaces in diffusion processes under the influence
  of strong absorption.
\newblock {\em Arch. Rational Mech. Anal.}, 149:183--212, 1999.

\bibitem{Grundy1}
R.~E. Grundy and L.~A. Peletier.
\newblock Short time behaviour of a singular solution to the heat equation with
  absorption.
\newblock {\em Proc. Roy. Soc. Edinburgh Sect. A}, 107:271--288, 1987.

\bibitem{Grundy2}
R.~E. Grundy and L.~A. Peletier.
\newblock The initial interface development for a reaction-diffusion equation
  with power-law initial data.
\newblock {\em Quarterly journal of mechanics and applied mathematics},
  43:535--559, 1990.

\bibitem{HerreroPierre}
M.~A. Herrero and M.~Pierre.
\newblock The {C}auchy problem for $u_t=\delta u^m$ when $0<m<1$.
\newblock {\em Transactions of the American Mathematical Society},
  291:145--158, 1985.

\bibitem{Herrero_Vazquez}
M.~A. Herrero and J.~L. Vazquez.
\newblock On the propagation properties of a nonlinear degenerate parabolic
  equation.
\newblock {\em Communications in Partial Differential Equations},
  7(12):1381--1402, 1982.

\bibitem{HerreroVazquez1}
M.~A. Herrero and J.~L. Vazquez.
\newblock Thermal waves in absorbing media.
\newblock {\em Journal of Differential Equations}, 74:218--233, 1988.

\bibitem{Ishige}
K.~Ishige.
\newblock On the existence of solutions of the {C}auchy problem for a doubly
  nonlinear parabolic equation.
\newblock {\em SIAM Journal on Mathematical Analysis}, 27(5):1235--1260, 1996.

\bibitem{Ivanov2}
A.~V. Ivanov.
\newblock H\"{o}lder estimates for equations of slow and normal diffusion type.
\newblock {\em Journal of Mathematical Sciences}, 85(1):1640--1644, 1997.

\bibitem{Ivanov1}
A.~V. Ivanov.
\newblock Regularity for doubly nonlinear parabolic equations.
\newblock {\em Journal of Mathematical Sciences}, 83(1):22--37, 1997.

\bibitem{Kalashnikov4}
A.~S. Kalashnikov.
\newblock The influence of absorption on the propagation of heat in a medium
  with heat conductivity that depends on the temperature.
\newblock {\em Zh. Vychisl. Mat. i Mat. Fiz.}, 16:689--696, 1976.

\bibitem{Kalashnikov3}
A.~S. Kalashnikov.
\newblock On a nonlinear equation appearing in the theory of non-stationary
  filtration.
\newblock {\em Trud. Semin. I. G. Pertovski}, 4:137--146, 1978.

\bibitem{Kalashnikov2}
A.~S. Kalashnikov.
\newblock On the propagation of perturbations in the first boundary value
  problem of a doubly-nonlinear degenerate parabolic equation.
\newblock {\em Trud. Semin. I. G. Pertovski}, 8:128--134, 1982.

\bibitem{Kalashnikov1}
A.~S. Kalashnikov.
\newblock Some problems of the qualitative theory of non-linear degenerate
  second-order parabolic equations.
\newblock {\em Russian Mathematical Surveys}, 42(2):169--222, 1987.

\bibitem{KPV}
S.~Kamin, L.~A. Peletier, and J.~L. Vazquez.
\newblock A nonlinear diffusion-absorption equation with unbounded initial
  data.
\newblock pages 243--263, 1992.

\bibitem{Kersner1}
R.~Kersner.
\newblock Degenerate parabolic equations with general nonlinearities.
\newblock {\em Nonlinear Analysis}, 4:1043--1062, 1980.

\bibitem{Leibenson}
L.~S. Leibenson.
\newblock General problem of the movement of a compressible fluid in porous
  medium.
\newblock {\em Izv. Akad. Nauk SSSR, Geography and Geophysics}, IX:7--10, 1945.

\bibitem{Oleinik}
O.~A. Oleinik, A.~S. Kalashnikov, and Ch.Y.Lin.
\newblock Cauchy problem and boundary value problems for an equation of
  nonstationary filtration.
\newblock {\em Izv. Akad. Nauk SSSR, Ser. Mat.}, 22:667--704, 1958.

\bibitem{Vespri1}
M.~M. Porzio and V.~Vespri.
\newblock H\"{o}lder estimates for local solutions of some doubly nonlinear
  degenerate parabolic equations.
\newblock {\em Journal of Differential Equations}, 103(1):146--178, 1993.

\bibitem{shmarev2015interfaces}
S.~Shmarev, V.~Vdovin, and A.~Vlasov.
\newblock Interfaces in diffusion-absorption processes in nonhomogeneous media.
\newblock {\em Mathematics and Computers in Simulation}, 118:360--378, 2015.

\bibitem{Tsutsumi}
M.~Tsutsumi.
\newblock On solutions of some doubly nonlinear degenerate parabolic equations
  with absorption.
\newblock {\em Journal of Mathematical Analysis and Applications},
  132(1):187--212, 1988.

\bibitem{Vasquez1}
J.~L. Vazquez.
\newblock The interfaces of one-dimensional flows in porous media.
\newblock {\em Transactions of the American Mathematical Society},
  285:717--737, 1984.

\bibitem{Vazquez2}
J.~L. Vazquez.
\newblock {\em The Porous Medium Equation: Mathematical Theory}.
\newblock Oxford Science Publications. Oxford University Press, 2007.

\bibitem{zeldovich}
Ya.~B. Zeldovich and A.~S. Kompaneets.
\newblock On the theory of heat propagation for temperature dependent thermal
  conductivity, in collection commemorating the 70th anniv. of {A. F.} {I}offe.
\newblock {\em Izdat. Akad. Nauk SSSR}, 1950.

\end{thebibliography}
\end{document}